\documentclass[11pt]{article}

\usepackage{amsmath,graphicx,verbatim, amsfonts,amssymb,amsthm,pictexwd,comment, amscd}
\usepackage{subfig}
\usepackage{graphics}
\theoremstyle{plain}
\date{}

\newcommand{\Z}{\mathbb{Z}}
\newcommand{\N}{\mathbb{N}}

\newcommand{\Q}{\mathbb{Q}}
\newcommand{\R}{\mathbb{R}}

\newcommand{\Int}{\mbox{Int}}
\newtheorem{theorem}{Theorem}[section]
\newtheorem*{theorem*}{Theorem}
\newtheorem{corollary}[theorem]{Corollary}
\newtheorem{lemma}[theorem]{Lemma}

\theoremstyle{definition}
\newtheorem{definition}[theorem]{Definition}
\newtheorem{remark}[theorem]{Remark}
\newtheorem{conjecture}[theorem]{Conjecture}
\newtheorem*{conjecture*}{Conjecture}
\newtheorem{calculation}[theorem]{Calculation}

\begin{document}
\title{The Tetrahedral Twists}
\author{Ren Yi}
\maketitle

\begin{abstract} 
We introduce a family of piecewise isometries $f_s$ parametrized by $s \in [0,1)$ on the surface of a regular tetrahedron, which we call the tetrahedral twists. This family of maps is similar to the PETs constructed by Patrick Hooper. We study the dynamics of the tetrahedral twists through the notion of renormalization. By the assistance of computer, we conjecture that the renormalization scheme exists on the entire interval $[0,1)$. In this paper, we show that this system is renormalizable in the subintervals $[\frac{53}{128},\frac{29}{70}]$ and $[\frac{1}{2},1)$.
\end{abstract}

\section{Introduction}
\quad Piecewise isometries have rich dynamical phenomena and they sometimes produce fractal-like pictures. To define these maps, let $X$ be a subset of $\R^n$ with a finite partition $\mathcal P=\{P_1, \cdots, P_n\} (n\geq 1)$. A piecewise isometry $T: X \to X$ is a map such that the restriction of $T$ to each $\Int(P_i), i=1, \cdots,n$ is a Euclidean isometry. The map is not defined on the boundaries $\partial P_i$ for $i=1,\cdots, n$. In this paper, we introduce a one-parameter family of piecewise isometries called the \textit{tetrahedral twists}. The intuitive definition is the following:  Let $\Delta$ be the surface of a regular tetrahedron of side length 1. Pick one pair of the opposite edges of $\Delta$ and cut them open, then $\Delta$ becomes a cylinder intrinsically. Rotate the cylinder by amount $s \in [0,1)$ counterclockwise. Glue the opposite edges so that $\Delta$ becomes the surface of a tetrahedron again. Apply this procedure on the other two pairs of opposite edges of $\Delta$. The entire process defines a piecewise isometry on $\Delta$ which is called the tetrahedral twist. 

\begin{definition}
A polytope exchange transformation (PET) is a piecewise isometry $T: X \to X$ on a polytope $X$ with two conditions:
\begin{enumerate}
\item The restriction on each $\Int(P_i)$ is a translation. 
\item The image $T(X)$ has the full area in $X$. 
\end{enumerate}
\end{definition} 

The tetrahedral twist maps are not piecewise translations. However, there exist double covers $(\tilde \Delta, \pi)$ of $\Delta$ such that the liftings of the tetrahedral twists produce PETs which we call the tetrahedral PETs. We will discuss this construction in Section 2.1.

\begin{definition}
Let $Y$ be a subset of $X$. Given a map $f: X \to X$, the \textit{first return} $f|_Y: Y \to Y$ is a map assigns every point $x \in Y$ to the first point in the forward orbit of $x$ lies in $Y$ under $f$, i.e.
\[
f|_Y(x) = f^k(x) \quad \mbox{where $k=\min\{f^k(x) \in Y\}  \quad k \geq 0$}.
\]
\end{definition}

A \textit{renormalization} of a PET $T: X \to X$ is the choice of a subset $Y$ of X such that the first return map $T|_Y : Y \to Y$ is also a PET.  The existence of renormalization scheme in a dynamical system allows us to study the acceleration of the orbits. If a renormalization scheme exists, we say that the system is \textit{renormalizable}. 
\begin{conjecture*}
The family of tetrahedral PETs is renormalizable in the parameter space $[0,1)$.
\end{conjecture*}

Out main goal is to show the following theorem:
\begin{theorem*}
The family of tetrahedral PETs has a renormalization scheme when the input parameter lies in the subintervals $I_1=[1/2,1)$ and $I_2=[53/128, 29/70]$. The subinterval $I_2$  of $[0,1)$ is a neighborhood of the irrational number $\sqrt 2 - 1$.
\end{theorem*}

\subsection{Background}
The interval exchange transformations (IETs) are the examples of piecewise isometries in dimension 1, see ~\cite{K}, ~\cite{Y} for surveys. The paper ~\cite{H} introduces rectangle exchanges which are the products of IETs. Another important class of piecewise isometries is piecewise rotation on polygons, which is studied in papers such as ~\cite{L}, ~\cite{LKV}. We know that piecewise rotations are closely related to the study of PETs in the following sense: Let $X \subset \R^2$ be a polygon together with a finite partition $\mathcal P$. If the piecewise rotation map $T$ performs a translation or rotation by a rational multiple of $\pi$ restricted on each element of $\mathcal P$, then there is a PET $S: Y \to Y$ conjugate to $T$ by a covering map $c: Y \to X$. 

The outer billiard maps on a convex polygon $P$ also give rise to piecewise isometries, see ~\cite{S4} for reference. The square of the outer billiards map is a piecewise translation outside $P$. In the paper ~\cite{S2}, ~\cite{S3}, a higher dimensional PET is constructed from the compactification of the outer billiard outside a kite. 

For work concerning renormalization of piecewise isometries, the Rauzy induction ~\cite{R} introduces a renormalization theory for IETs. In the paper ~\cite{L}, a general theory of renormalization of piecewise rotations is developed. The paper ~\cite{S2} shows that the renormalization scheme exists for PETs arising from the outer billiards on Penrose kites. 

The tetrahedral twists are very similar to the PETs described in ~\cite{C1}, ~\cite{C2} by Hooper. In ~\cite{C1}, the map is defined on four copies of torus, which we denote by $X$. For every point $x \in X$, the map performs a translation in either horizontal direction parametrized by $\alpha \in [0,\frac{1}{2})$ or in the vertical direction parametrized by $\beta \in [0,\frac{1}{2})$, respectively. This is the first example of PETs in 2-dimensional parameter space which is invariant under renormalization. Hooper describes the renormalization procedure in terms of the renormalization of the Truchet tilings, see ~\cite{C3}. 

\section{Definition of the Tetrahedral twists}
Let $\gamma_1, \gamma_2, \gamma_3$ be reflections about the points $a_1=(-1,0)$, $a_2=(-\frac{1}{2},\frac{\sqrt 3}{2})$ and $a_3=(0,0)$, respectively. Let $G$ be the group generated by $\gamma_1, \gamma_2, \gamma_3$. Define the space $\Delta = \R^2 / G$.
A fundamental domain for the action of $G$ by reflection is the union of four equilateral triangles $A_0, A_1,A_2,A_3$ of side length 1 where $A_0$ has the vertices $a_1,a_2,a_3$ and $A_n$ is the reflection of $A_0$ by the line connecting the points $a_n$ and $a_{n+1 \text{ (mod 3)}}$ for $n=1,2,3$. In fact, these four triangles are the faces of a regular tetrahedron. 
\begin{figure}[h]
\begin{center}
\includegraphics[width=1.4in]{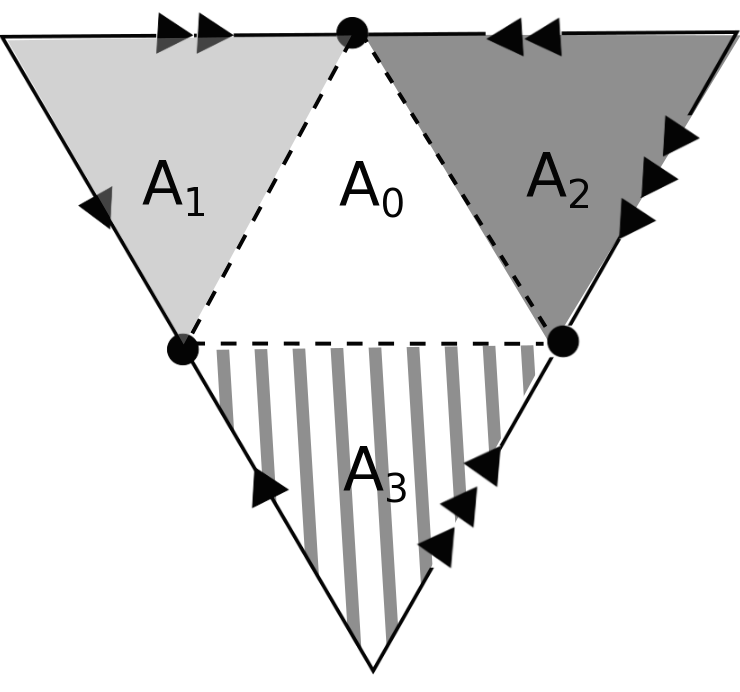}
\caption{A fundamental domain $\Delta$ for the action of $G$.}
\end{center}
\end{figure}

Fix three parameters $s_0,s_1,s_2 \in [0,1)$. Let $\omega_0=(1,0)$, $\omega_1 = (-\frac{1}{2}, \frac{\sqrt 3}{2})$, $\omega_2=(-\frac{1}{2}, -\frac{\sqrt 3}{2})$ be the unit vectors in the directions of cube roots of unity. For each $i=0,1$ or $2$, we define the map $f_{i}: \Delta \rightarrow \Delta$ as follows: 
\[
f_{i}(x,y) = (x, y) +  \sigma_is_i \cdot  \omega_i \mod G
\]
where
\[
\sigma_i = \left\{
\begin{array}{l l}
-1  & \quad \mbox{if $(x,y) \in A_{3-i}$}\\
1 & \quad \mbox{otherwise.}
\end{array} \right.
\]
%\begin{figure}[h]
%\centering
%\includegraphics[width=1.4 in]{action1.png}
%\includegraphics[width=1.42 in]{action2.png}
%\includegraphics[width=1.4 in]{action3.png}
%\caption{Examples of the map $f_{s_0},f_{s_1}$ and $f_{s_2}$, from left to right.}
%\end{figure}
\begin{figure}[h]
\centering
\includegraphics[scale=0.7]{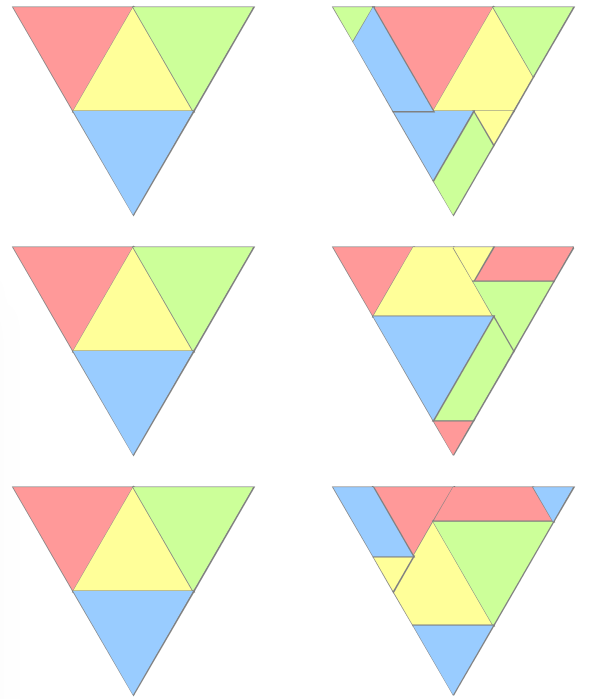}
\caption{Examples of the map $f_{s_0},f_{s_1}$ and $f_{s_2}$ from top to bottom.}
\end{figure}
The maps $f_0, f_1, f_2$ are illustrated in figure $2$. Now, suppose $s=s_i$ for $i=0,1,2$. A family of tetrahedral twists $f_s: X \rightarrow X$ is defined as
\[
 f_s(x,y)= f_{2} \circ f_{1}\circ f_{0}(x,y).
\]

\subsection{Connection to PETs} 
Let $\Lambda$ be the lattice generated by two vectors $(2,0)$ and $(1,-\sqrt 3)$ and $\tilde \Delta$ be the torus $\R^2/\Lambda$. Let $\pi: \tilde \Delta \to \Delta$ be the projection given by
\[
\pi(x,y)=\left\{
\begin{array}{ll}
(x,y) \quad &\mbox{if $y\geq \sqrt 3 x$}\\
(-x,-y) \quad &\mbox{otherwise},
\end{array}
\right.
\]
and $(\tilde \Delta, \pi)$ is a double cover of $\Delta$. Let $\iota: \tilde \Delta \to \tilde \Delta$ be the reflection about the origin. Define $\tilde A_\alpha=\pi^{-1}(A_\alpha)$ for $\alpha=0,\cdots, 3$, so 
$\tilde A_\alpha=A_\alpha \cup \iota(A_\alpha)$ for each $\alpha$. For fixed parameters $s_0,s_1, s_2 \in [0,1)$, the lifting $\tilde f_{i}$ of the map $f_{i}$ is given by the following equation:
\[
 \tilde f_{i}(x,y)= (x,y)+\sigma_i s_i \cdot \omega_i \mod \Lambda,
\] 
where 
\[
\sigma_i=\left\{
\begin{array}{l l}
1 & \quad \mbox{if $(x,y) \in \displaystyle\bigcup_{\alpha \neq 3-i} A_{\alpha} \cup \iota(A_{3-i})$}\\
-1 & \quad \mbox{otherwise}.
\end{array}
\right.
\]
\begin{figure}[h]
\centering
\includegraphics[scale=0.7]{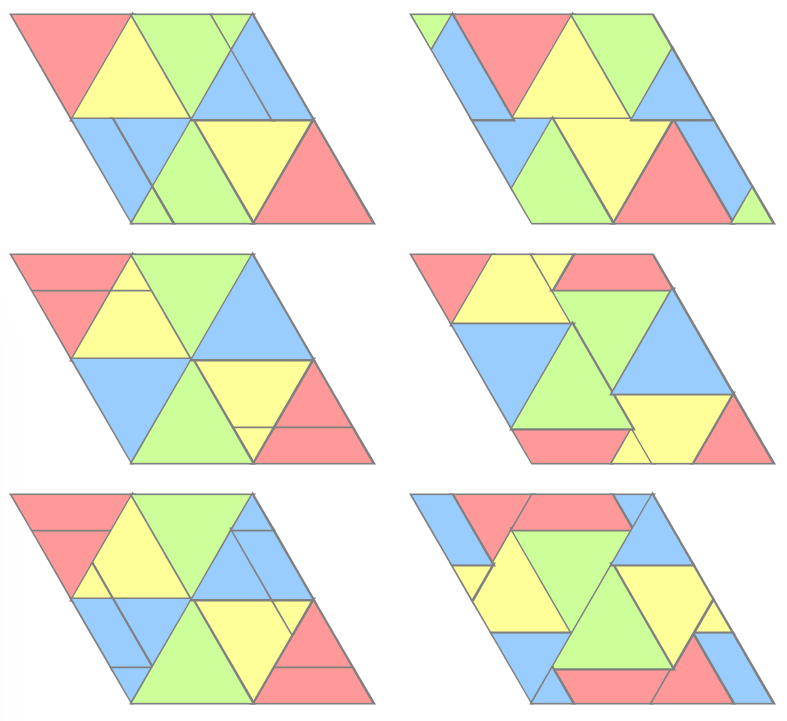}
\caption{This illustrates the PETs $\tilde f_{s_0}$, $\tilde f_{s_1}$ and $f_{s_2}$} on $\tilde \Delta$.
\end{figure}

For each $i\in \{0,1,2\}$,  $\tilde \Delta$ is divided into halves by a line $l_i$  through the origin in direction $\omega_i$.  On one half of $\tilde \Delta$, the map $\tilde f_{i}$ translates every point $(x,y)$  by amount $s_i$ (mod $\tilde \Lambda$) in direction $\omega_i$. In the other half, every point is translated by the same amount but in the opposite direction $-\omega_i$.  Therefore, $\tilde f_{i}$ is a PET, for each $i=0,1,2$. 

As mentioned in previous section, we set $s=s_i$ for all $i=0,1,2$. The composition $\tilde f_s: \tilde \Delta \to \tilde \Delta$ is defined as
\[
\tilde f_s(x,y) = \tilde f_2 \circ \tilde f_1 \circ \tilde f_0(x,y).
\]
For $s\in (0,1)$, the map $\tilde f_s$ is a PET. We call $\tilde f_s$ a \textit{tetrahedral PET}. More precisely, for every point $(x,y) \in \tilde \Delta$, there is some translation vector $V$ such that
\[
\tilde f_s(x,y)=(x,y) + V. 
\]
where $V$ is in the form of 
\[
(A+B\cdot \frac{s}{2}, C+D\cdot \frac{\sqrt 3}{2}s)
\]
for some $A,B,C,D \in \Z$.  

Fix $s \in (0,1)$. The partition $\mathcal D$ of $\tilde \Delta$ associated to the tetrahedral PET $\tilde f_s$ is obtained by the following fact:  Suppose that $g: X \to Y$ and $h: Y \to Z$ are  PETs and $\mathcal P=\{P_i\}_{i=1}^n$, $\mathcal Q=\{Q_j\}_{j=1}^m$ are the partitions of $X,Y$ determined by the maps $g$ and $h$, respectively. Then, $\{P_{ij}\}_{i=1, j=1}^{n,m}$ is a finer partition of $X$ determined by the PET $h \circ g: X \to Z$ where
\[
P_{ij}= P_i \cap g^{-1} (Q_j).
\]
The following figures show an example of  the tetrahedral PET $\tilde f_s$ when $s=5/13$. The figure on the left shows  the partition $\mathcal D$ determined by a tetrahedral PET $\tilde f_s$. The figure on the right shows the image of every element in $\mathcal D$ under $\tilde f_s$. To be clear, we assign a number to each element in the partition $\mathcal D$ in the left figure whose image is the shape with the same number in the right figure. 
\begin{figure}[h]
\centering
\includegraphics[scale=0.65]{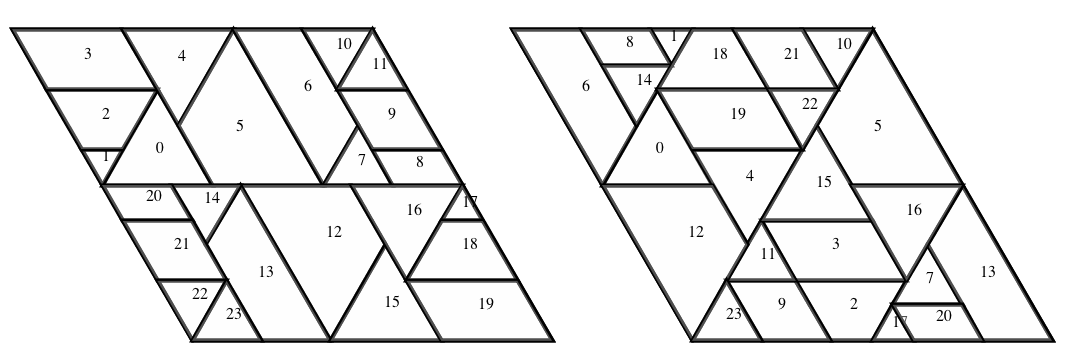}
\caption{An example of a tetrahedral PET $\tilde f_s$ for $s=5/13$}
\end{figure}

\subsection{Periodic Tilings}
Let $p\in \tilde \Delta \cap (\displaystyle\bigcup_{D\in \mathcal D} \partial D)^c$ be a periodic point with period $n$ of the map $\tilde f_s$. 
\begin{definition}
A periodic tile $\Diamond_p$ of $\tilde f_s$ is a maximal subset containing $p$ such that $\tilde f_s$ is entirely defined on $\Diamond_p$ and all points in $\Diamond_p$ have the same period as $p$.
\end{definition}
 For a given point $p \in \tilde \Delta$,  we provide a pseudo-code algorithm to produce a periodic tile $\Diamond_p$ containing $p$ of $\tilde f_s$.
\begin{enumerate}
\item Let $P_0$ be a polygon in the partition $\mathcal D$ such that $p\in \Int(P_0)$.
\item If $k<n$, then let $P_{k+1}=f(P_{k}) \cap D_{k+1}$ where $D_{k+1}$ is some element in the partition $\mathcal D$ and $f^{k+1}(p) \in D_{k+1}$. Set $k=k+1$
\item Else, return $P_{k}$. 
\end{enumerate}

By construction, every periodic tile $\Diamond_p$ is convex since it is the intersection of convex polygons. A periodic tiling $X_s$ is the union of all periodic tiles $\Diamond_p$ for all $p \in \tilde \Delta$.
\begin{figure}[h]
\centering
\includegraphics[width=2.8in]{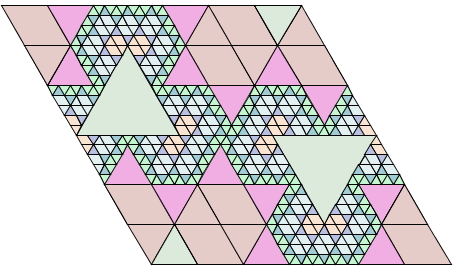}
\caption{Periodic tiling $X_s$  for $s=4/13$}
\end{figure}

\begin{figure}[h]
\centering
\includegraphics[width=2.8in]{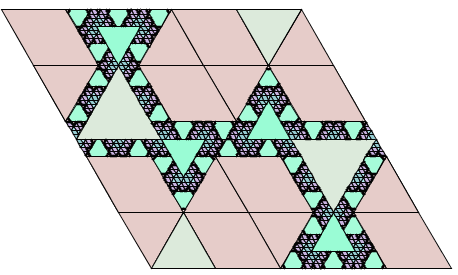}
\caption{Periodic tiling $X_s$  for $s=68/157$}
\end{figure}

\subsection{The Main Result: Partial Renormalization}
Define the renormalization map $R:[0,1)\to [0,1)$ by the formula.
\[
R(s)=\left\{
\begin{array}{l l}
\frac{s}{1-2s} - \lfloor \frac{s}{1-2s}\rfloor \quad & \mbox{if $[0,\frac{1}{2})$}\\
1-s & \mbox{if $x \in [\frac{1}{2},1)$}.
\end{array}
\right.
\]
For any subset $S\subset \tilde \Delta$, we write $\tilde f_s|_S$ as the first return map of $\tilde f_s$ on $S$. 

Let $s \in (0,1)$ and $\varsigma_s: \tilde \Delta \to \tilde \Delta$ be the reflection about the line $x= -\frac{t}{2}$. Define $H_s \subset \tilde \Delta$ as the semi-regular hexagon with vertices:
\[
V_0=(0,0),\quad  V_1=(1-s)(\frac{1}{2},\frac{\sqrt 3}{2}), \quad  V_2=(\frac{1}{2}-s,\frac{\sqrt 3}{2})
\]
\[
\varsigma_s(V_0), \quad  \varsigma_s(V_1), \quad \varsigma_s(V_2).
\]
Figure 6 shows an example of $H_s$ for $s=5/12$.  Define the subsets $Y'_s , Y_s$ of $\tilde \Delta$ as follows:
\[
Y_s'=A_1 \cup A_3 \cup H_s, \quad Y_s=Y_s' \cup \iota(Y_s').
\]

\begin{theorem}[Partial Renormalization]
Suppose $s \in [53/128,29/70]$ and $t=R(s)$. There exists a set $Z_s \subset \tilde \Delta$ such that 
\[
\tilde f_s|_{Z_s} = \phi_s^{-1} \circ \tilde f_t|_{Y_t} \circ \phi_s
\]
where $\phi_s: \tilde \Delta \to \tilde \Delta$ is a similarity with the scale factor $c=\frac{1}{1-2s}$. 
\end{theorem}

Let $\mathcal U$ be the upper half plane in of $\R^2$ and $\mathcal L$ be the lower half. Define
\[
\tilde \Delta_{\mathcal U} = \tilde \Delta \cap \mathcal U, \quad \tilde \Delta_{\mathcal L} = \tilde \Delta \cap \mathcal L.
\]
\begin{theorem}[Partial Renormalization]
If $s \in [1/2,1)$, then $\tilde f_s$ is conjugate to $\tilde f_{1-s}$ by a piecewise translation map $\phi_s: \tilde \Delta \to \tilde \Delta$ where
\[
\phi_s(x,y)=\left\{
\begin{array}{l l}
(x,y)+(\frac{1}{2},-\frac{\sqrt 3}{2}) \quad \mbox{if $y\geq 0$}\\
(x,y)+(-\frac{1}{2}, \frac{\sqrt 3}{2}) \quad \mbox{if $y<0$}.
\end{array}
\right.
\]
\end{theorem}
The theorem above says that the periodic tiling of $\tilde f_{s}$ and $\tilde f_{1-s}$ are same up to the interchange of $\tilde \Delta_{\mathcal U}$ and $\tilde \Delta_{\mathcal L}$. 

\begin{conjecture}[Renormalization]
For any $s\in [0,\frac{1}{2})$, there exists a set $Z_s \subset \tilde \Delta$ such that $\tilde f_t|_{Y_t}$ is conjugate to
$\tilde f_s|_{Z_s}$ via a similarity $\phi_s: \tilde \Delta \to \tilde \Delta$ with the scale $\frac{1}{1-2s}$.
\end{conjecture}

\begin{definition}
A space $X$ has a mostly self-similar structure if there is a  disjoint union $\displaystyle\bigsqcup_i^n X_i \subset X$ such that each $X_i$ is self-similar.
\end{definition}

\begin{corollary}
Let $s=\sqrt 2-1 \in [53/128, 29/70]$ that is a fixed point under the renormalization map $R$. The periodic tiling $X_s$ is mostly self-similar. 
\end{corollary}

The proofs will be provided in section 5. 

\begin{figure}
\centering
\includegraphics[scale=0.45]{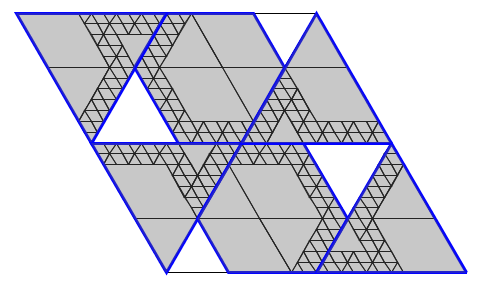}
\caption{ $Y_t$ lightly shaded for $t=5/12=R(29/70)$}
    \centering
    \subfloat{{\includegraphics[scale=0.4]{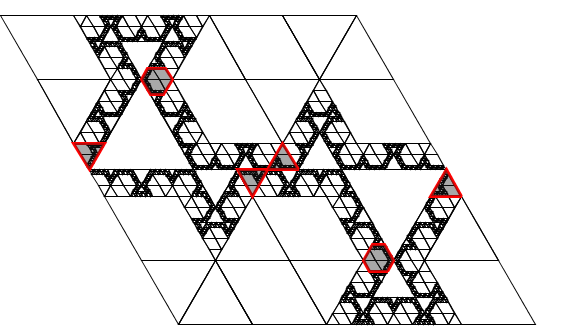} }}%
    
    \subfloat{{\includegraphics[scale=0.6]{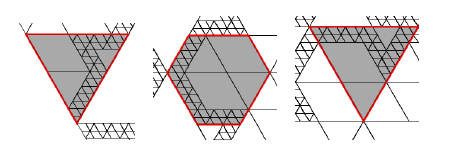}}}%
    \caption{$Z_s$ lightly shaded for $s=29/70$}
\vspace{2em}
\centering
\includegraphics[width=7cm]{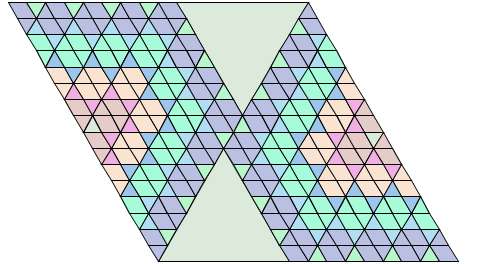}
\includegraphics[width=7.3cm]{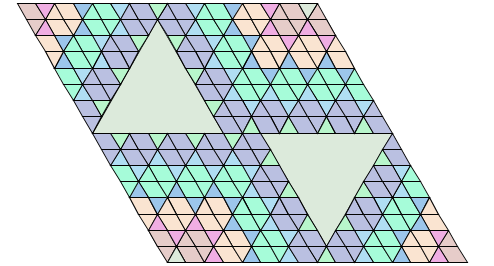}
\caption{The periodic tilings for $s=4/5$ on the left and $t=1/5=R(s)$ on the right}
\end{figure}

\begin{figure}
\centering
\includegraphics[width=3in]{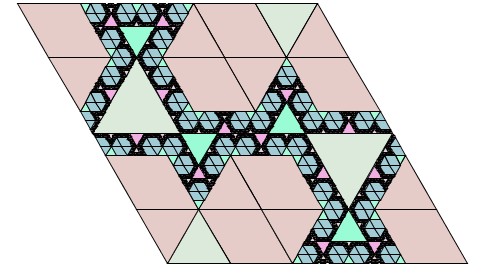}
\includegraphics[width=2.6in]{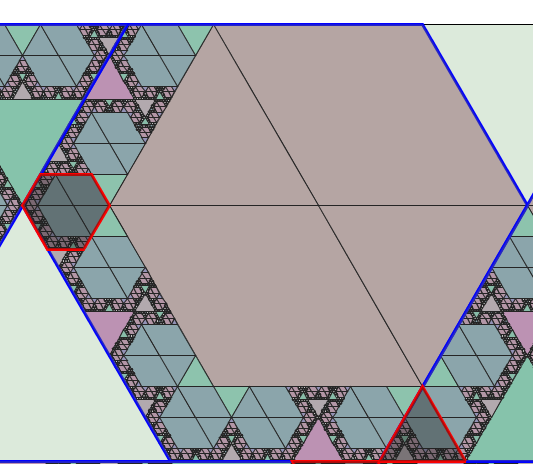}
\caption{Periodic tiling $X_s$ for $s=\sqrt 2-1$}
\end{figure}

\section{The Renormalization Map}
In this section, we explore the properties of the renormalization map $R$ and its connection to continued fraction expansions. Recall that $R:[0,1)\to [0,1)$ is given by the formula:
\[
R(s)=\left\{
\begin{array}{l l}
\frac{s}{1-2s} - \lfloor \frac{s}{1-2s}\rfloor \quad & \mbox{if $s\in [0,\frac{1}{2})$}\\
1-s \quad & \mbox{if $s \in [\frac{1}{2},1)$}.
\end{array}
\right.
\]
Each fixed point $s$ of $R$ in $(0,\frac{1}{2})$ is in the form of 
\[
s=\frac{-n+\sqrt{n(n+2)}}{2}, \quad \mbox{$n \geq 1, n\in \Z$}.
\]
Moreover, all fixed points $s$ have continued fractions expansion in the following form:
\[
(0; \overline{2,n}),  \quad \mbox{$n \geq 1, n\in \Z$}.
\]

\begin{lemma}
Let $\frac{p}{q} \in\Q$ in $[0,1)$. There exists some integer $k \geq 0$ such that $R^k(\frac{p}{q}) \in \{0,\frac{1}{2}\}$. 
\end{lemma}

\begin{proof}
We have 
\[
R\big(\frac{p}{q}\big)=\left\{
\begin{array}{l l}
\frac{p}{q-2p} \mod \Z \quad & \mbox{if $\frac{p}{q} \in [0,\frac{1}{2})$}\\
\frac{q-p}{q} \quad & \mbox{if $\frac{p}{q} \in [\frac{1}{2},1)$}.
\end{array}
\right.
\]
Write $\frac{p_k}{q_k}=R^k(\frac{p}{q})$. When we apply the square map $R^2$, the denominators have the fact that $q_{k+2}\leq q_{k}-2$. Thus, the $q_k$ drops to a value of $1$ or $2$ for some $k \geq 0$. It means that $R(\frac{p_k}{q_k})$ must be $0$ or $\frac{1}{2}$ for some $k \geq 0$.
\end{proof}

For every $s\in [0,1)$, we can define a coding map $\mathcal M: [0,1) \to \Z \times \Z \times \{\pm 1\}$ as follows:
\[
\mathcal M(s)= \left\{
\begin{array}{l l}
(2,n,1) \quad & \mbox{if $s\in (0,\frac{1}{2})$ and $n \leq \frac{s}{1-2s} < n+1$} \\
(2,0,1) \quad & \mbox{if $s=\frac{1}{2}$}\\
(0,1,-1) \quad & \mbox{if $s \in (\frac{1}{2},1)$}\\
(0,\infty,1) \quad & \mbox{if $s=0$}
\end{array}
\right.
\]

\begin{definition}
Let $s \in (0,1)$. A coding sequence $\{(m_k,n_k,r_k) \in \Z \times \Z \times \{\pm 1\}\}_{k=0}$ for $s$ is a sequence of finite or infinite length such that every element $(m_k,n_k,r_k)$ in the sequence is given by the formula
\[
(m_k, n_k, r_k)=\mathcal M(R^k(s))
\]
for $k \geq 0$. If $R^k(s)=0$, then the sequence terminates at step at step $k-1$. If $R^k(s)=\frac{1}{2}$, then the sequence terminates at step $k$.
\end{definition}
By lemma 3.1, the coding sequence for a rational number terminates after a finite number of steps. For example, the coding sequence of $5/23$ is 
\[
\{(2,0,1), (2,1,1), (0,1,-1), (2,1,1)\}.
\]

Now, we define  $\Omega$ as the set of all sequence $\{(m_k, n_k,r_k)\}$ of finite or infinite length satisfying the following condition:
\begin{enumerate}
\item $(0,1,-1)$ cannot appear consecutively in the sequence
\item For a finite sequence $\{(m_k,n_k,r_k)\}_{k=1}^{l}$, the last element $(m_l, n_l,r_l) \neq (0,1,-1)$.
\end{enumerate}

\begin{theorem}
Let $\{(m_k,n_k,r_k)\}$ be a sequence in $\Omega$.
\begin{enumerate} 
\item If $\{(m_k,n_k,r_k)\}$ is infinite, there is a unique $s \in [0,1)$ determined via the formula 
\[
s=\dfrac{1}{m_0+\dfrac{1}{n_0+\dfrac{r_0}{m_1+\dfrac{1}{n_1+\dfrac{r_1}{m_2+\cdots}}}}}.
\]
Moreover, the coding sequence of $s$ is $\{(m_k,r_k,n_k)\}_{k=0}^\infty$. 
\item If the sequence $\{m_k,n_k,r_k\}_{k=0}^l$ is finite of length $l+1$ and $(m_l,n_l,r_l)\neq(2,0,1)$, then there is a unique $s \in [0,1)$ determined by 
\[
s=\dfrac{1}{m_0+\dfrac{1}{n_0+\cdots +\dfrac{r_l}{m_l+\dfrac{1}{n_l}}}}. \quad \tag{*}
\]
\item If $\{(m_k,n_k,r_k)\}^l_{k=0}$ is finite and $(m_l,n_l,r_l)=(2,0,1)$, then there is a unique $s \in (0,1)$ determined by the formula (*) but without $1/n_l$.
\end{enumerate} 
\end{theorem}

\begin{proof}
Let $\alpha=\{(m_k,n_k,r_k)\}_{k=0}^\infty$ be a sequence of elements in $\Omega$ and $s \in [0,1)$ is determined by the formula (*). We want to show that $s$ has the coding sequence $\alpha$.
\begin{enumerate} 
\item Suppose the first element in the sequence $\{m_k, n_k, r_k\}_{k=0}^\infty$ is $(2,n_0,1)$ for $n_0 \in \N$. Write $t=R(s)$ i.e.,
\[
t=\frac{s}{1-2s}-n_0 \quad \mbox{for some $n_0 \in \Z^+$}.
\]
By computation, we have 
\[
s=\dfrac{n_0+t}{1+2(n_0+t)}=\dfrac{1}{2+\dfrac{1}{n_0+t}}.
\]
\item Suppose $(m_0,n_0,r_0)=(0,1,-1)$. Similarly, we set $t=R(s)$ so that $t=1-s$. Therefore, 
\[
s=1-t=\dfrac{1}{0+\dfrac{1}{1-t}}.
\]
\end{enumerate}
Repeat this argument by substituting $s=R(s)$. If there exists some element $(m_k, n_k,r_k)=(0,0,1)$ or $(2,0,1)$ in the sequence, then the sequence terminates at the length $k$. We obtain the desired statement.
\end{proof}

\begin{definition}
Let $s \in (0,1)$ and  $\{(m_k,n_k,r_k) \in \Z \times \Z \times \{\pm 1\} \}_{k\geq 0}$ be the coding sequence of $s$. A  splitted expansion of $s$ is defined as follows:  
\begin{itemize}
\item[--] $a_0 = m_0, \quad a_1 = n_0$,
\item[--] $a_{2k}= m_{k}\displaystyle\prod_{i=0}^{k-1} r_i, \quad a_{2k+1} = n_{k}\displaystyle\prod_{i=0}^{k-1} r_i$ \quad for $k>0$.
\item[--] If $s$ is rational and the  coding sequence of $s$ has length $n+1$, then the splitted expansion terminates at $a_{2n-2}$ if $R^n(s)=0$ or at $a_{2n-1}$ if $R^n(s)=\frac{1}{2}$.
\end{itemize}
\end{definition}

For example, the splitted sequence for $s=5/23$ is 
\[
(0; 2,0,2,1,0,1,-2,-1).
\]
Here are several observations of the splitted expansion:
\begin{itemize}
\item If $s$ is a fixed point by $R$, then the splitted expansion is same as  the continued fraction expansion of $s$ which is in the form of 
\[
(0;\overline{2,n}), \quad \mbox{for every integer $n \geq 1$}.
\]
\item The splitted expansion is shifted by 2 digits to the left under the renormalization map $R$. 
\item To translate between the splitted expansion and the signed continued fraction expansion, we have to replace the fragment $\cdots, a_{j-1}, 0, a_{j+1}, a_{j+2}, \cdots$ with $\cdots, a_{j-1}+a_{j+1},a_{j+2},\cdots$. 

For example, $\frac{45}{178}$ has the splitted fraction expansion $(0; 2,0,0,1,-2,-21,-2)$ and its signed continued fraction expansion is $(0; 2,1,-2,-21,-2)$.
\end{itemize}

Suppose $s$ has a splitted expansion $(0; a_1, a_2, \cdots)$ of inifite length. We set the $k$th convergent $(0;a_1,a_2, \cdots, a_k)$ of $s$ as
\[
c_k=\frac{p_k}{q_k}.
\]
The recurrent formulas for $p_k$ and $q_k$ are same to the ones of the continued fraction expansion, i.e.
\begin{itemize}
\item $p_0=0, q_0=1$.
\item If $a_1\neq 0$, then we set $p_1=1, q_1=a_1$ and 
\begin{eqnarray*}
p_m &=& a_m p_{m-1}+p_{m-2}\\
q_m &=& a_m q_{m-1}+q_{m-2}, \quad \mbox{for $m \geq 2, m \in \N$.}
\end{eqnarray*}
\item If $a_1=0$, then $a_2 \neq 0$. We can set $p_1=0, q_1=1$ and $p_2=1, q_2=a_2$. Then $p_m$ and $q_m$ are obtained by the same formula as above for all integer $m\geq 3$.
\end{itemize}
The theorem below says that the splitted expansion gives us a good approximation of irrationals.

\begin{lemma}
Let $s \in [0,1)$ be irrational with infinite splitted expansion $(0;a_1, a_2, \cdots)$. If $\frac{p_k}{q_k} \to s$ as $k \to \infty$, then $|\frac{p_k}{q_k}-s|\leq \frac{6}{q_k^2}$. 
\end{lemma} 
The proof is same as Theorem 11.4 in ~\cite{S1} by passing to the signed continued fraction expansion of $s$.

\section{The Fiber Bundle Picture}
The motivation of this section is to construct convex polyhedra and reduce all the calculations to the polyhedra, which is very similar to Schwartz's construction  in ~\cite{S1}. Recall that $\tilde \Delta$ obtained by gluing the parallelogram with vertices 
\[
\pm(-3/2, \sqrt 3/2), \pm(1/2,\sqrt 3/2).
\]
We define 
\[
\mathcal X = \{(x,y,s)| (x,y)\in \tilde \Delta, s \in [0,1]\}.
\]
AS a fiber bundle over $[0,1]$. The fiber above $s$ is the parallelogram $\tilde \Delta$. 
Define the fiber bundle map $F:\mathcal X \to \mathcal X$ as
\[
(x,y,s) \mapsto (\tilde f_s(x,y),s).
\]
Define $\mathcal X(I)$ as the set
\[
\{(x,y,s)| (x,y,s) \in \mathcal X, s \in I\}.
\]
 It is useful to split the fiber bundle $\mathcal X$ as 
\[
\mathcal X = \mathcal X[0,1/2] \cup \mathcal X[1/2, 1].
\]

\subsection{Maximal Domains in $\mathcal X([\frac{1}{2},1))$}
\begin{definition}
A maximal domain of $\mathcal X[\frac{1}{2},1)$ is a maximal subset of $\mathcal X[\frac{1}{2},1)$ such that the bundle map $F$ is entirely defined and continuous.
\end{definition}
For $s \in [\frac{1}{2},1)$, every cross section of the union of maximal domains in $\mathcal X[\frac{1}{2},1)$ at the plane $z=s$ is the partition of $\tilde \Delta$ determined by the tetrahedral PET $\tilde f_s$.  By the assistant of computer, we know that $\mathcal X[\frac{1}{2},1)$ is partitioned into 22 maximal domains. Each maximal domain is a convex polyhedron which has rational vertices. Experimentally, we obtain the fact that every maximal domain in $\mathcal X[\frac{1}{2},1)$ has vertices in the form of 
\[
(\frac{a}{2q}, \frac{b \sqrt 3}{2q}, \frac{p}{q})
\]
for $\frac{p}{q}\in \{1,\frac{1}{2}\}$ and integers $a,b \in (-4q,4q)$. 

\subsection{Maximal domains in $\mathcal Y([0,\frac{1}{2}])$} 
Let $\mathcal A_1, \mathcal A_3, \mathcal H$ be subsets of $\mathcal X[0,\frac{1}{2}]$ whose fiber over $s$ are the sets $A_1, A_3$ and $H_s$, respectively.  Define the reflection $\iota: \mathcal X \to \mathcal X$ as  
\[
\iota(x,y,s)=\iota(-x,-y,s).
\]
Let $\mathcal Y', \mathcal Y$ be the sets 
\[
\mathcal Y' =  \mathcal A_1 \cup \mathcal A_3 \cup \mathcal H, \quad \mathcal Y = \mathcal Y' \cup \iota(\mathcal Y').
\]
The set $\mathcal Y(I)$ is defined similarly as $\mathcal X(I)$, which is a fiber bundle over $t \in I$ such that the fiber above $t$ is $Y_t$.

Now, we consider the maximal domains in $\mathcal X(I)$ for $I \subset (0,\frac{1}{2})$.
\begin{definition}
Let $I$ be a subinterval of $[0,\frac{1}{2}]$. Let $S$ be any one of the six polyhedra in $\mathcal Y(I)$. A maximal domain in $\mathcal Y(I)$ is a maximal subset where the first return $F|_S$ on $S$ is entirely defined and continuous. 
\end{definition}

\begin{figure}[h]
\centering
\includegraphics[scale=0.5]{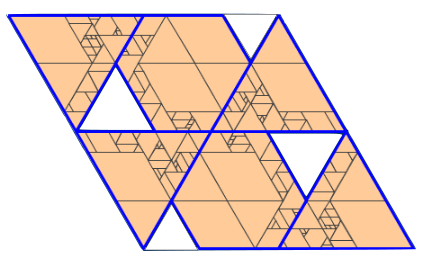}
\caption{A union cross sections of maximal domains in $\mathcal X[\frac{7}{17},\frac{5}{12}]$ at the plane $z=12/29$}
\end{figure}

For any subinterval $I \subset [0,1]$, if the number of maximal domains in $\mathcal Y(I)$ is finite, we can apply the calculation on the vertices of the maximal domains to show the conjugacy of the first return maps. However, the number of maximal domains is not always finite on each arbitrary subintervals of $[0,\frac{1}{2}]$.  The next experimental result  provides a classification of subintervals $I$ of $[\frac{1}{3},\frac{1}{2}]$ such that the number of maximal domains in $\mathcal Y(I)$ is finite. 
\begin{calculation}
Let $I$ be a subinterval of $[\frac{1}{3},\frac{1}{2})$. If $I$ is in one of the following form of continued fraction expansion indexed by $m,n \in \N$, then number of maximal domains in $\mathcal Y(I)$ is fixed. Furthermore, none of the maximal domains vanishes in the interval $I$. 
\begin{enumerate}
\item $A_{m,n}=[(0;2,m,n), \quad (0;2,m,n-1)]$,\qquad $m\geq 2, n\geq 2$
\item $B_{m,n}=[(0;2,m,1,n),\quad (0;2,m,1,n+1)]$, \qquad $m \geq 2, n\geq 1$
\item $C_{m,n}=[(0;2,1,m,n),\quad (0;2,1,m,n+1]]$, \qquad $m \geq 1$ odd, $n\geq 1$
\item $D_{m,n}=[(0;2,1,m,n),(0;2,1,m,n+1)]$,\qquad $m\geq 2$ even, $n\geq 2$
\item $E_{m,n}=[(0;2,1,m,1,n),\quad (0;2,1,m,1,n-1)]$, \qquad $m \geq 2$ even, $n \geq 2$. 
\end{enumerate}
\end{calculation}

\subsection{Notation}
For convenience, we introduce some notation in this section. 

Define $s_{m,n}, t_{m,n}$ and the interval $A_{m,n}, \bar A_{m,n}$ as follows:
\[
s_{m,n}=(0;2,2,2,m,n), \quad t_{m,n}=R(s_{m,n})=(0;2,m,n),
\]
\[
A_{m,n}=[s_{m,n},s_{m,m-1}], \quad \bar A_{m,n}=[t_{m,n},t_{m,n-1}].
\]
Then, we denote $A_{m\geq 2,n\geq 3}, A_{2, n\geq 3}$ to be the union
\[
\bigcup_{m\geq 2, n\geq 3} A_{m,n} = [\frac{12}{29},\frac{5}{12}], \quad \bigcup_{n\geq 3} A_{2,n} = [\frac{12}{29},\frac{29}{70}],
\]
respectively. Similarly, denote $\bar A_{m\geq 2,n\geq 3}, \bar A_{2,n\geq 3} $ as 
\[
 \bigcup_{m\geq 2,n\geq 3} \bar A_{m,n}=[\frac{2}{5},\frac{1}{2}], \quad  \bigcup_{n\geq 3} \bar A_{2,n}= [\frac{2}{5},\frac{5}{12}],
\]
respectively. 

\subsection{Maximal domains in $\mathcal Y(\bar A_{2,3})$}
Note that
\[
A_{2,3}=[\frac{41}{99}, \frac{29}{70}] \quad \mbox{and} \quad \bar A_{2,3}=R(A_{2,3})=[\frac{7}{17},\frac{5}{12}].
\]
$\mathcal Y(\bar A_{2,3})$ is partitioned into 176 maximal domains, each of which is a convex  polytope. Figure 10 shows cross sections of the union maximal domains in $\mathcal X[7/17, 5/12]$ at the plane $z=12/29$. By calculation, the vertices of every maximal domain are in the form of
\[
(\frac{a}{2q}, \frac{b\sqrt 3}{2q}, \frac{p}{q})
\]
where $\frac{p}{q} \in \{\frac{7}{17}, \frac{5}{12}\}$ are end points of the interval $I$ and $a,b \in (-4q, 4q)$ are integers.

\begin{lemma}
For each connected component $S \in \mathcal Y(\bar A_{2,3})$,  $F|_S$ is a piecewise affine map.
\end{lemma}

\begin{proof}
For each point $(x,y,s) \in \mathcal Y(\bar A_{2,3})$, we have 
\[
F|_S(x,y,s)=(x,y,s)+(A+B\frac{s}{2},C+D\frac{s \sqrt 3}{2}, 0)
\]
where $A,B,C,D \in \Z$. If we vary the point $(x,y,s)$ in a neighborhood of $(x,y,s)$, the integers $A,B,C,D$ do not change.  Since $\mathcal Y(\bar A_{2,3})$ is partitioned into finitely many maximal domains, $F|_S$ is a piecewise affine map on $S$. 
\end{proof}

\begin{definition}
A maximal domain $P$ in $\mathcal Y(\bar A_{m
\geq 2,n\geq 3})$ is a permanent maximal polyhedron if $P$ satisfies the following condition: 
\begin{itemize}
\item At least one vertex of $P$ has $z$-coordinate $2/5$,
\item  At least one vertex of $P$ has $z$-coordinate $1/2$.
\end{itemize}
\end{definition}

\begin{definition}
A maximal domain $P$ in $\mathcal Y(\bar A_{m \geq 2,n\geq 3})$ is called resident maximal polyhedron if $P$ satisfies the following condition: 
\begin{itemize}
\item At least one vertex of $P$ has $z$-coordinate $2/5$,
\item  At least one vertex of $P$ has $z$-coordinate $5/12$,
\item All the vertices $v=(x,y,z)$ of $P$ has $2/5 \leq z \leq 5/12$.
\end{itemize}
\end{definition}

It's equivalent to say that if a maximal domain $P$ in $\mathcal Y(\bar A_{m \geq 2, n\geq 3})$ does not vanish between the plane $z=\frac{2}{5}$ and $z=\frac{1}{2}$, then $P$ is a permanent polyhederon. Note that $2/5, 1/2$ are two end points of the interval $A_{m\geq 2, n\geq 3}$. Moreover, if a maximal domain $P \subset \mathcal Y(\bar A_{2, n \geq 3})$ lies between the plane $z=\frac{2}{5}$ and $z=\frac{5}{12}$ and the intersection of $P$ with each plane is non-empty, then $P$ is a resident maximal polyhedron in $\mathcal Y(\bar A_{m\geq 2, n\geq 3})$. These notations help us to classify the maximal domains restricting to the smaller intervals $\bar A_{2,3}$.

\begin{definition}
If a maximal domain $P \subset \mathcal Y(\bar A_{2,3})$ is obtained by chopping from a resident maximal domain $Q$ in $\mathcal Y(\bar A_{m\geq  2, n \geq 3})$, we say $P$ is a primary maximal domain in $\mathcal Y(\bar A_{2,3})$. More precisely, $P$ is primary if 
\[
P = \{(x,y,z): (x,y,z) \in Q \mbox{ and }z \in \bar A_{2,3}\}, 
\]
for some resident maximal domain $Q$ in $\mathcal Y(\bar A_{m\geq 2,n \geq 3})$. 
\end{definition}

In $\mathcal Y(\bar A_{2,3})$, there are 176 maximal domains,  where 150 are primary. Let  $\mathcal M_1$ be the set of resident maximal polyhedra in $\mathcal Y(\bar A_{m\geq 2, n\geq 3})$ and $\mathcal M_1(\bar A_{2,3})$ be the set of primary maximal domains in $\mathcal Y(\bar A_{2,3})$.  Denote $\mathcal M_2$ to be the set of rest 26 maximal polyhedra in $\mathcal Y(\bar A_{m \geq 2,n\geq 3})$ which also produce maximal domains in $\mathcal Y(\bar A_{2,3})$. These are the polyhedra lying strictly above the plane $z=\frac{2}{5}$. We list them in the last section of the paper. 

\section{Proof of the Main Theorem}
Before going to the proof, we provide the explicit formula of the similarity $\phi_s: \tilde \Delta \to \tilde \Delta$ which appeared in the renormalization Theorem 2.1.
\[
\phi_s(x,y) = \left\{
\begin{array}{l l}
c(x+1,y)+(-1,0) &\quad \mbox{if $(x,y) \in A_1$}\\
c(x-1,y)+(1,0) &\quad \mbox{if $(x,y) \in \iota(A_1)$}\\
c(x,y) &\quad \mbox{if $(x,y,z) \in  A_3, \iota(A_3)$}\\
c(x+\frac{1+s}{2},y+ \frac{\sqrt 3(1-s)}{2})-(\frac{1+t}{2}, \frac{\sqrt 3(1-t)}{2}) &\quad \mbox{if $(x,y,z) \in  H_t$}.\\
c(x-\frac{1+s}{2},y- \frac{\sqrt 3(1-s)}{2})+(\frac{1+t}{2}, \frac{\sqrt 3(1-t)}{2}) &\quad \mbox{if $(x,y,z) \in \mathcal \iota(H_t)$}.
\end{array}
\right.
\] 
where the scalar $c=\frac{1}{1-2s}$.
Then, we can define the set $Z_s$ in theorem 2.1 as 
\[
Z_s=\phi^{-1}_s(Y_t)
\]
and $\mathcal Z(I)$ be the fiber bundle over $I$ such that the fiber above $s \in [0,1]$ is $Z_s$. A maximal domain in $\mathcal Z(I)$ is defined in the same way as the maximal in $\mathcal Y(I)$. Moreover, a maximal domain $Q$ in $\mathcal Z(A_{m\geq 2,n\geq 3})$ is a permanent maximal polyhedron if $Q$ satisfies the following properties:
\begin{itemize}
\item $Q$ has at least one vertex with $z$-coordinate $12/29$,
\item $Q$ has at least one vertex with $z$-coordinate $5/12$.
\end{itemize}

We say $Q$ a resident maximal polyhedron in $\mathcal Z(A_{m \geq 2, n \geq 3})$, if 
\begin{itemize}
\item $Q$ has at least one vertex with $z$-coordinate $12/29$ 
\item $Q$ has at least one vertex with $z$-coordinate $29/70$.
\item The $z$-coordinates of all vertices of $Q$ should satisfy $12/29 \leq z \leq 29/70$.
\end{itemize}

Let $\mathcal N_1$ be the collection of resident maximal polyhedron in $\mathcal Z(A_{m\geq 2,n\geq 3}) $. By direct computation, there are 162 maximal domains in $\mathcal Z(A_{2,3})$, 136 of which are chopped from resident maximal domains in $\mathcal Z(A_{2,n\geq 3})$. Let us denote the set of primary maximal domains by $\mathcal N_1(A_{2,3})$. Moreover, there are $26$  maximal polyhedra in $\mathcal Z(A_{m\geq 2,n\geq 3})$ from which the non-primary maximal domains in $\mathcal Z(A_{2,3})$ can be obtained. Denote the set of these 26 non-resident maximal polyhedra in $\mathcal Z(A_{m\geq 2, n\geq 3})$ by $\mathcal N_2$.

\subsection{Renormalization on the subinterval $A_{2,3}$}
The goal in this section is to show that for all $s \in A_{2,3}$, $\tilde f_s|_{Y_t}$ and $\tilde f_t|_{Z_s}$ are conjugate by the similarity map $\phi_s$. To prove this, we've attached 1-dimensional parameter space to the planar torus $\tilde \Delta$ and want to apply the calculation in $\R^3$. For calculation, we always refer to open polyhedra.

First, we piece together the similarities $\phi_s$ on $\tilde \Delta$ to construct a piecewise affine map $\phi: \mathcal X \to \mathcal X$ on the fiber bundle which is defined as 
\[
\phi(x,y,t)=(\phi_s(x,y),\frac{s}{1-2s}-2).
\]
\begin{lemma}
Fix a parameter $s\in A_{2,3}=[\frac{41}{99},\frac{29}{70}]$. Let $t=R(s)$. The first return map $\tilde f_s|_{Z_s}$ satisfies
\[
\phi_s \circ \tilde f_s|_{Z_s} = \tilde f_t|_{Y_t} \circ \phi_s
\]
\end{lemma}

\proof \textit{Step 1.} For every non-resident maximal polyhedron $P_i$, $i=1,\cdots, 26$ in $\mathcal M_1$, we check that $P_i$ satisfies the  following properties: 
\begin{enumerate}
\item There exists a non-resident maximal domain $Q_j$ in $\mathcal N_2$ such that 
\[
P_i = \phi(Q_j).
\]
It follows that there is a one-to-one correspondence between the elements in $\mathcal M_2$ and the ones in $\mathcal N_2$. For computation, it is sufficient to check that the set of vertices of $Q_j$ are $\{\phi(V_0), \cdots, \phi(V_n)\}$ where $V_0, \cdots, V_n$ are the vertices of $P_i$.
\item Let $S$ be a polyhedra in $\mathcal Y(\bar A_{m\geq 2,n\geq 3})$ such that $P_i \subset S$, then $\phi(S) \subset \mathcal Z(A_{m\geq 2, n\geq 3})$. Moreover, the maximal domains satisfy the following condition: 
\[
\phi^{-1} \circ F|_{S}(P_i) \subset F|_{\phi^{-1}(S)}(Q_j).
\]
\item We denote the polyhedra $\phi^{-1} \circ F|_{S}(P_i)$ by $P_{ij}$ if $P_i$ satisfies the inclusion above. We check the fact: 
\[
\Int(P_{ij} )\cap \Int(P_{i'j'}) =\emptyset \quad \mbox{for each pair of $i \neq i'$}.
\]
\item
\[
\sum \mbox{Volume}(P_{ij})=\sum \mbox{Volume}(Q_j), \quad \mbox{for $i,j=1,\cdots 26$}.
\]
\end{enumerate}

\textit{Step 2.} Next, we consider the points in resident maximal polyhedra. We apply the calculation on the set of resident maximal polyhedra because if the conjugacy is satisfied, then it follows that the renormalization scheme exists for all points in $Z_s$ when $s \in A_{2,3}$. Since there is no one-to-one correspondence between the resident maximal domains in $\mathcal M_1$ and the ones in $\mathcal N_1$, we cannot apply the same calculation as before.  However, by computer assistance, we find that each element in $\mathcal M_1$ is a subpolyhedron of a resident maximal polyhedron in $\mathcal N_2$ up to a similarity.  We apply the similar calculations as in Step 1 and check the following properties:
\begin{enumerate}
\item For every resident maximal polyhedron $P_i \in $ in $\mathcal M_1, i=1,\cdots, 150$, there exists a resident maximal polyhedron $Q_j \in \mathcal N_1$ such that $P_i \subset \phi(Q_j)$ for some $j\in \{1,\cdots, 136\}$.
\item Let $S$ be the connected component of $\mathcal Y(\bar A_{m\geq 2, n\geq 3})$ such that $P_i \subset S$. Denote $P'_i$ as the polyhedron 
\[
\phi^{-1} \circ F|_{S} (P_i)
\]
$P'_i$ satisfies that 
\[
P'_i \subset F|_{\phi^{-1}(S)} (Q_j).
\]
\item 
\[
\Int(P'_i) \cap \Int(P'_j) = \emptyset \quad \mbox{for $i \neq j$}.
\]
\item 
\[
\sum_{i=1}^{150} \mbox{Volume}(P'_{i}) = \sum_{j=1}^{136} \mbox{Volume}(Q_j).
\]
\end{enumerate}
Hence, we've shown that the map $\tilde f_s$ is renormalizable when $s \in A_{2,3}=[\frac{41}{99},\frac{29}{70}]$.  \qed

\begin{remark}
Since the size of the data is too large to include in this paper, I provide the code on my website and one can check the data of all the resident maximal polyhedra from my website. The URL is: $math.brown.edu/ \sim renyi$ .

\end{remark}

\subsection{Renormalization on the interval $A_{2,4}$}
\begin{lemma}
Lemma 4.4 holds for parameter $s \in A_{2,4}=[\frac{53}{128}, \frac{29}{70}]$.
\end{lemma}

We want to apply the same method as used in the previous case. Therefore, we need to classify the maximal domains in $\mathcal Y(\bar A_{2,4})$ and $\mathcal Z(A_{2,4})$ first. The primary maximal domains in $\mathcal Y(\bar A_{2,4})$ are the maximal domains chopped from the resident maximal polyhedera  defined in section 4.4. 

The non-primary maximal domains of $\bar Y(A_{2,4})$ are either obtained by chopping from the elements in $\mathcal M_2$ or they are the newly-appeared maximal domains defined as follows:

\begin{definition}
A maximal domain $P$ in $\mathcal Y(\bar A_{m\geq 2,n\geq 3})$ is newly-appeared at the parameter $s$ if it satisfies the following:
\begin{itemize}
\item $s=\max\{z: (x,y,z)\in P\}$.
\item The number of vertices in $P$ with $z$-coordinate being $s$ is less than 3. 
\end{itemize}
\end{definition}
If a maximal domain $P$ is newly-appeared, then $P$ lies below the plane $z=s$ and it can only touches the plane $z=s$ at a point or a line segment. The primary and newly-appeared maximal domains of $\mathcal Z(A_{2,n\geq 3})$ at $s$ are defined similarly by replacing $t_{m,n}$ with $s_{m,n}$.

 By computation, $\mathcal Y(\bar A_{2,4})$ is partitioned into 178 maximal domains and 150 of them are primary. Then the set $\mathcal Z(A_{2,4})$ has 136 primary and 28 non-primary maximal domains. Since we've shown that the renormalization exists for all points in every resident maximal polyhedron, we are left to check the points in non-primary maximal domains. Among the 28 non-primary maximal domains in $\mathcal Z(A_{2,4})$ (or $\mathcal Y(\bar A_{2,4})$), there are 12 of them are obtained by chopping from non-resident maximal domains in $\mathcal Z(A_{m\geq 2, n\geq 3})$ (or $\mathcal Y(\bar A_{m\geq 2,n\geq 3})$), which we have already done the calculation. 

If $P$ is a non-primary maximal domain in $\mathcal Y(\bar A_{2,4})$ but does not belong to the above case, then $P$ must be chopped from a newly-appeared polyhedron at the parameter $t_{2,3}$. This is because the maximal z-coordinate of all points in $P$ must be $t_{2,3}$. Moreover, if $P$ has more than 2 vertices with $z=t_{2,3}$, then $P$ is either primary or inherited from a maximal polyhedron appeared in $Y(A_{2,3})$. The same argument works for the case of $\mathcal Z(A_{2,4})$. The lists of all newly-appeared maximal domains in $\mathcal Y(\bar A_{2,4})$ and $\mathcal Z(A_{2,4})$ are provided in section 6.

There is a one-to-one correspondence between the newly-appeared maximal polyhedra in in $\mathcal Y(A_{m\geq 2, n\geq 3})$at $t_{m,n}$ and the ones in $\mathcal Z(A_{m\geq 2,n\geq 3})$. We apply the same calculation  as in  lemma 5.1. Therefore, we show that when the parameter $s \in [\frac{53}{128},\frac{41}{99}]$, it is true that
\[
\tilde f_s|_{Z_s}= \phi_s \circ \tilde f_t|_{Y_t} \circ \phi^{-1}_s.
\]  

\subsection{Renormalization on the interval $[\frac{1}{2},1)$}
We want to show that $\tilde f_s$ is conjugate to $\tilde f_{1-s}$ when $s\in [\frac{1}{2},1)$ by a piecewise translation $\phi_s$. Recall that $\phi_s$ is the map interchanging the upper half and lower half of the torus $\tilde \Delta$. Therefore, we can piece together the map $\phi_s$ for $s \in [\frac{1}{2},1)$ to get an affine map $\phi: \mathcal X \to \mathcal X$ in $\R^3$: 
\[
\phi(x,y,s)=(\phi_s(x,y),1-s).
\]
It is easy to see that the affine map $\phi$ is an involution as well. As discussed in Section 4.1, there is a partition $\mathcal P=\{\displaystyle P_i\}_{i=1}^{22}$ of $\mathcal X([\frac{1}{2},1))$ such that each $P_i$ is a maximal domains in $\mathcal X([\frac{1}{2},1))$ determined by the fiber bundle map $F: \mathcal X \to \mathcal X$ 
\[
(x,y,s) \mapsto (\tilde f_s(x,y),s).
\]
There is a partition $\mathcal Q=\{Q_j\}_{j=1}^{24}$ of $\mathcal X[0,\frac{1}{2}]$ such that each $Q_j$ is a maximal subset of $\mathcal X[0,\frac{1}{2}]$ where $F$ is entirely defined and continuous. Next, we construct a finer partition $\mathcal Q'$ of $\mathcal X([0,\frac{1}{2}])$ as follows: 
\begin{itemize}
\item If there exists some $Q_i\in \mathcal Q$ such that $\phi(Q_i) \subset P_j$, then the polyhedron $Q_i$ is an element in $\mathcal Q'$.
\item If there exists some $Q_i \in \mathcal Q$ such that $\phi(Q_i) \supset P_j$, then the polyhedron $\phi(P_j)$ is an element in $\mathcal Q'$.
\end{itemize}
$\mathcal Q'$ is partitioned into 26 elements and the bundle map $F$ is well-defined on each $Q'_{k} \in \mathcal Q'$. Then, we check that the following properties hold:
\begin{enumerate}
\item  For each $Q'_k \in \mathcal Q'$, there exists some $P_i \in \mathcal P$ such that
\[
\phi(Q'_k) \subset P_i.
\]
\item 
\[
\phi \circ F(Q'_{k}) \subset  F(P_i).
\]
\item  
\[
\mbox{Int}(F(Q'_{k})) \cap \mbox{Int}(F(Q'_{l}))=\emptyset, \quad \mbox{if $k\neq l$}.
\] 
\item 
\[
\sum_{j=1}^{26}\mbox{volume}(F(Q'_k))=\sum_{i=1}^{22}\mbox{volume}(P_i).
\]
\end{enumerate}
Thus, we show that the tetrahedral PET $\tilde f_s$ on $\tilde \Delta$ is renormalizable when $s \in [\frac{1}{2},1)$.

\section{The Computational Data}
The 26 non-secondary maximal polyhedron of $\mathcal Y(\bar A_{2,3})$ are listed as follows: 
{\tiny
\[
P_0=\begin{pmatrix}
-1/8 \\ 1/24 \cdot \sqrt 3\\ 5/12
\end{pmatrix}
\begin{pmatrix}
-1/6 \\ 0 \\ 5/12
\end{pmatrix}
\begin{pmatrix}
-2/17 \\ 1/17 \cdot \sqrt 3\\ 7/17
\end{pmatrix}
\begin{pmatrix}
-3/17 \\ 1/17 \cdot \sqrt 3\\ 7/17
\end{pmatrix}, \quad \iota(P_0)
\]

\[
P_1 = \begin{pmatrix}
-2/3 \\ 1/3 \cdot \sqrt 3\\ 5/12
\end{pmatrix}
\begin{pmatrix}
-5/8 \\ 3/8 \cdot \sqrt 3\\ 5/12
\end{pmatrix}
\begin{pmatrix}
-7/12 \\ 1/3 \cdot \sqrt 3\\ 5/12
\end{pmatrix}
\begin{pmatrix}
-11/17 \\ 6/17 \cdot \sqrt 3\\ 7/17
\end{pmatrix}, \quad \iota(P_1)
\]

\[
P_{2}=\begin{pmatrix}
-11/12 \\ 5/12 \cdot \sqrt 3\\ 5/12
\end{pmatrix}
\begin{pmatrix}
-5/6 \\ 5/12 \cdot \sqrt 3\\ 5/12
\end{pmatrix}
\begin{pmatrix}
-16/17 \\ 7/17 \cdot \sqrt 3\\ 7/17
\end{pmatrix}
\begin{pmatrix}
-31/34 \\ 13/34 \cdot \sqrt 3\\ 7/17
\end{pmatrix}, \quad \iota(P_{2})
\]

\[
P_{3}=\begin{pmatrix}
-2/3 \\ 5/12 \cdot \sqrt 3\\ 5/12
\end{pmatrix}
\begin{pmatrix}
-17/24 \\ 11/24 \cdot \sqrt 3\\ 5/12
\end{pmatrix}
\begin{pmatrix}
-11/17 \\ 7/17 \cdot \sqrt 3\\ 7/17
\end{pmatrix}
\begin{pmatrix}
-12/17 \\ 7/17 \cdot \sqrt 3\\ 7/17
\end{pmatrix}, \quad \iota(P_{3})
\]

\[
P_{4}=\begin{pmatrix}
-2/3 \\ 5/12 \cdot \sqrt 3\\ 5/12
\end{pmatrix}
\begin{pmatrix}
-17/24 \\ 11/24 \cdot \sqrt 3\\ 5/12
\end{pmatrix}
\begin{pmatrix}
-3/4 \\ 5/12 \cdot \sqrt 3\\ 5/12
\end{pmatrix}
\begin{pmatrix}
-12/17 \\ 7/17 \cdot \sqrt 3\\ 7/17
\end{pmatrix}, \quad \iota(P_{4})
\]

\[
P_{5}=\begin{pmatrix}
3/8 \\ 1/24 \cdot \sqrt 3\\ 5/12
\end{pmatrix}
\begin{pmatrix}
5/12 \\ 1/12 \cdot \sqrt 3\\ 5/12
\end{pmatrix}
\begin{pmatrix}
6/17 \\ 1/17 \cdot \sqrt 3\\ 7/17
\end{pmatrix}
\begin{pmatrix}
13/34 \\ 1/34 \cdot \sqrt 3\\ 7/17
\end{pmatrix}, \quad \iota(P_{5})
\]

\[
P_{6}=\begin{pmatrix}
3/8 \\ 1/24 \cdot \sqrt 3\\ 5/12
\end{pmatrix}
\begin{pmatrix}
1/3 \\ 1/12 \cdot \sqrt 3\\ 5/12
\end{pmatrix}
\begin{pmatrix}
5/12 \\ 1/12 \cdot \sqrt 3\\ 5/12
\end{pmatrix}
\begin{pmatrix}
6/17 \\ 1/17 \cdot \sqrt 3\\ 7/17
\end{pmatrix}, \quad \iota(P_{6})
\]

\[
P_7=\begin{pmatrix}
-11/12 \\ 5/12 \cdot \sqrt 3\\ 5/12
\end{pmatrix}
\begin{pmatrix}
-33/34 \\ 13/34 \cdot \sqrt 3\\ 7/17
\end{pmatrix}
\begin{pmatrix}
-16/17 \\ 7/17 \cdot \sqrt 3\\ 7/17
\end{pmatrix}
\begin{pmatrix}
-31/34 \\ 13/34 \cdot \sqrt 3\\ 7/17
\end{pmatrix}
\begin{pmatrix}
-21/22 \\ 4/11 \cdot \sqrt 3\\ 9/22
\end{pmatrix}
\begin{pmatrix}
-43/44 \\ 17/44 \cdot \sqrt 3\\ 9/22
\end{pmatrix}, \quad \iota(P_7)
\]

\[
P_8=\begin{pmatrix}
-17/24 \\ 11/24 \cdot \sqrt 3\\ 5/12
\end{pmatrix}
\begin{pmatrix}
-12/17 \\ 7/17 \cdot \sqrt 3\\ 7/17
\end{pmatrix}
\begin{pmatrix}
-11/17 \\ 7/17 \cdot \sqrt 3\\ 7/17
\end{pmatrix}
\begin{pmatrix}
-23/34\\ 15/34 \cdot \sqrt 3\\ 7/17
\end{pmatrix}
\begin{pmatrix}
-15/22 \\ 9/22 \cdot \sqrt 3\\ 9/22
\end{pmatrix}
\begin{pmatrix}
-7/11 \\ 9/22 \cdot \sqrt 3\\ 9/22
\end{pmatrix}, \quad \iota(P_8)
\]

\[
P_9= \begin{pmatrix}
-2/3\\ 1/3 \cdot \sqrt 3\\ 5/12
\end{pmatrix}
\begin{pmatrix}
-7/12 \\ 1/3 \cdot \sqrt 3\\ 5/12
\end{pmatrix}
\begin{pmatrix}
-11/17 \\ 6/17 \cdot \sqrt 3\\ 7/17
\end{pmatrix}
\begin{pmatrix}
-21/34 \\ 11/34 \cdot \sqrt 3\\ 7/17
\end{pmatrix}
\begin{pmatrix}
-29/44 \\ 15/44 \cdot \sqrt 3\\ 9/22
\end{pmatrix}
\begin{pmatrix}
-15/22 \\ 7/22 \cdot \sqrt 3\\ 9/22
\end{pmatrix}, \quad \iota(P_9)
\]

\[
P_{10}=\begin{pmatrix}
-1/8 \\ 1/24 \cdot \sqrt 3\\ 5/12
\end{pmatrix}
\begin{pmatrix}
-5/34 \\ 3/34 \cdot \sqrt 3\\ 7/17
\end{pmatrix}
\begin{pmatrix}
-2/17 \\ 1/17 \cdot \sqrt 3\\ 7/17
\end{pmatrix}
\begin{pmatrix}
-3/17 \\ 1/17 \cdot \sqrt 3\\ 7/17
\end{pmatrix}
\begin{pmatrix}
-2/11 \\ 1/11 \cdot \sqrt 3\\ 9/22
\end{pmatrix}
\begin{pmatrix}
-3/22 \\ 1/11 \cdot \sqrt 3\\ 9/22
\end{pmatrix}, \quad \iota(P_{10})
\]

\[
P_{11}=\begin{pmatrix}
5/12 \\ 1/12 \cdot \sqrt 3\\ 5/12
\end{pmatrix}
\begin{pmatrix}
6/17 \\ 1/17 \cdot \sqrt 3\\ 7/17
\end{pmatrix}
\begin{pmatrix}
13/34 \\ 1/34 \cdot \sqrt 3\\ 7/17
\end{pmatrix}
\begin{pmatrix}
7/17 \\ 1/17 \cdot \sqrt 3\\ 7/17
\end{pmatrix}
\begin{pmatrix}
4/11 \\ 1/22 \cdot \sqrt 3\\ 9/22
\end{pmatrix}
\begin{pmatrix}
17/44 \\ 1/44 \cdot \sqrt 3\\ 9/22
\end{pmatrix}, \quad \iota(P_{11})
\]

\[
P_{12}=\begin{pmatrix}
1/4 \\ 1/6 \cdot \sqrt 3\\ 5/12
\end{pmatrix}
\begin{pmatrix}
7/24 \\ 1/8 \cdot \sqrt 3\\ 5/12
\end{pmatrix}
\begin{pmatrix}
3/17 \\ 3/17 \cdot \sqrt 3\\ 7/17
\end{pmatrix}
\begin{pmatrix}
4/17 \\ 3/17 \cdot \sqrt 3\\ 7/17
\end{pmatrix}
\begin{pmatrix}
2/11 \\ 2/11 \cdot \sqrt 3\\ 9/22
\end{pmatrix}
\begin{pmatrix}
7/44 \\ 7/44 \cdot \sqrt 3\\ 9/22
\end{pmatrix}, \quad \iota(P_{12}).
\]
}

Here are 26 non-secondary maximal polyhedron of $\mathcal Z(A_{2,3})$.
{\tiny
\[
Q_0=\begin{pmatrix}
-17/28 \\ 1/4 \cdot \sqrt 3\\ 29/70
\end{pmatrix}
\begin{pmatrix}
-43/70 \\ 17/70 \cdot \sqrt 3\\ 29/70
\end{pmatrix}
\begin{pmatrix}
-20/33 \\ 25/99 \cdot \sqrt 3 \\ 41/99
\end{pmatrix}
\begin{pmatrix}
-61/99 \\ 25/99 \cdot \sqrt 3 \\ 41/99
\end{pmatrix}, \quad \iota(Q_0)
\]

\[
Q_1=\begin{pmatrix}
-2/3 \\ 1/3\cdot \sqrt 3 \\ 5/12
\end{pmatrix}
\begin{pmatrix}
-5/8 \\ 3/8 \cdot \sqrt 3 \\ 5/12
\end{pmatrix}
\begin{pmatrix}
-7/12 \\ 1/3 \cdot \sqrt 3 \\ 5/12
\end{pmatrix}
\begin{pmatrix}
-11/17 \\ 6/17 \cdot \sqrt 3 \\ 7/17
\end{pmatrix}, \quad \iota(Q_1)
\]

\[
Q_2 = \begin{pmatrix}
-69/70 \\ 1/14 \cdot \sqrt 3 \\ 29/70
\end{pmatrix}
\begin{pmatrix}
-34/35 \\ 1/14 \cdot \sqrt 3\\ 29/70
\end{pmatrix}
\begin{pmatrix}
-98/99 \\ 7/99 \cdot \sqrt 3\\ 41/99
\end{pmatrix}
\begin{pmatrix}
-65/66 \\ 13/198 \cdot \sqrt 3\\ 41/99
\end{pmatrix}, \quad \iota(Q_2)
\]

\[
Q_3=\begin{pmatrix}
-33/35 \\ 5/12 \cdot \sqrt 3 \\ 5/12
\end{pmatrix}
\begin{pmatrix}
-17/24 \\ 11/24 \cdot \sqrt 3\\ 5/12
\end{pmatrix}
\begin{pmatrix}
-12/17 \\ 7/17 \cdot \sqrt 3\\ 7/17
\end{pmatrix}
\begin{pmatrix}
-11/17 \\ 7/17 \cdot \sqrt 3\\ 7/17
\end{pmatrix}, \quad \iota(Q_3)
\]

\[
Q_4=\begin{pmatrix}
-33/35 \\ 1/14 \cdot \sqrt 3 \\ 29/70
\end{pmatrix}
\begin{pmatrix}
-19/20 \\ 11/140 \cdot \sqrt 3 \\ 29/70
\end{pmatrix}
\begin{pmatrix}
-67/70 \\ 1/14 \cdot \sqrt 3 \\ 29/70
\end{pmatrix}
\begin{pmatrix}
-94/99 \\ 7/99 \cdot \sqrt 3 \\ 41/99
\end{pmatrix}, \quad \iota(Q_4)
\]

\[
Q_5=\begin{pmatrix}
9/140 \\ 1/140 \cdot \sqrt 3 \\ 29/70
\end{pmatrix}
\begin{pmatrix}
1/14 \\ 1/70 \cdot \sqrt 3 \\ 29/70
\end{pmatrix}
\begin{pmatrix}
2/33 \\ 1/99 \cdot \sqrt 3 \\ 41/99
\end{pmatrix}
\begin{pmatrix}
13/198 \\ 1/198 \cdot \sqrt 3 \\ 41/99
\end{pmatrix}, \quad \iota(Q_5)
\]

\[
Q_6=\begin{pmatrix}
9/140 \\ 1/140 \cdot \sqrt 3\\ 29/70
\end{pmatrix}
\begin{pmatrix}
2/35 \\ 1/140\cdot \sqrt 3 \\ 29/70
\end{pmatrix}
\begin{pmatrix}
1/14 \\ 1/70 \cdot \sqrt 3 \\ 29/70
\end{pmatrix}
\begin{pmatrix}
2/33 \\ 1/99 \cdot \sqrt 3 \\ 41/99
\end{pmatrix}, \quad \iota(Q_6)
\]

\[
Q_7=\begin{pmatrix}
-69/70 \\ 1/14 \cdot \sqrt 3\\ 29/70
\end{pmatrix}
\begin{pmatrix}
-197 \\ 13/198 \cdot \sqrt 3\\ 41/99
\end{pmatrix}
\begin{pmatrix}
-98/99 \\ 7/99 \cdot \sqrt 3\\ 41/99
\end{pmatrix}
\begin{pmatrix}
-65/66 \\ 13/198 \cdot \sqrt 3\\ 41/99
\end{pmatrix}
\begin{pmatrix}
-127/128 \\ 1/16 \cdot \sqrt 3\\ 53/128
\end{pmatrix}
\begin{pmatrix}
-255/256 \\ 17/256 \cdot \sqrt 3 \\ 53/128
\end{pmatrix}, \quad \iota(Q_7)
\]

\[
Q_8=\begin{pmatrix}
-19/20 \\ 11/140 \cdot \sqrt 3\\ 29/70
\end{pmatrix}
\begin{pmatrix}
-94/99 \\ 7/99 \cdot \sqrt 3 \\ 41/99
\end{pmatrix}
\begin{pmatrix}
-31/33 \\ 7/99 \cdot \sqrt 3 \\ 41/99
\end{pmatrix}
\begin{pmatrix}
-17/18 \\ 5/66 \cdot \sqrt 3 \\ 41/99
\end{pmatrix}
\begin{pmatrix}
-121/128 \\ 9/128 \cdot \sqrt 3 \\ 53/128
\end{pmatrix}
\begin{pmatrix}
\-15/16 \\ 9/128 \cdot \sqrt 3 \\ 53/128
\end{pmatrix}, \quad \iota(Q_8)
\]

\[
Q_9=\begin{pmatrix}
-7/10 \\ 3/10 \cdot \sqrt 3 \\ 29/70
\end{pmatrix}
\begin{pmatrix}
-24/35 \\ 3/10 \cdot \sqrt 3 \\ 29/70
\end{pmatrix}
\begin{pmatrix}
-23/33 \\ 10/33 \cdot \sqrt 3 \\ 41/99
\end{pmatrix}
\begin{pmatrix}
-137/198 \\ 59/198 \cdot \sqrt 3 \\ 41/99
\end{pmatrix}
\begin{pmatrix}
-179/256 \\ 77/256 \cdot \sqrt 3 \\ 53/128
\end{pmatrix}
\begin{pmatrix}
-45/64 \\ 19/64 \cdot \sqrt 3 \\ 53/128
\end{pmatrix}, \quad \iota(Q_9).
\]

\[
Q_{10}=\begin{pmatrix}
-17/28 \\ 1/4 \cdot \sqrt 3\\ 29/70 
\end{pmatrix}
\begin{pmatrix}
-11/18 \\ 17/66 \cdot \sqrt 3 \\ 41/99
\end{pmatrix}
\begin{pmatrix}
-20/33 \\ 25/99 \cdot \sqrt 3 \\ 41/99
\end{pmatrix}
\begin{pmatrix}
-61/99 \\ 25/99 \cdot \sqrt 3 \\ 41/99
\end{pmatrix}
\begin{pmatrix}
-79/128 \\ 33/128 \cdot \sqrt 3 \\ 53/128
\end{pmatrix}
\begin{pmatrix}
-39/64 \\ 33/128 \cdot \sqrt 3 \\ 53/128
\end{pmatrix}, \quad \iota(Q_{10})
\]

\[
Q_{11}=\begin{pmatrix}
1/14 \\ 1/70 \cdot \sqrt 3 \\ 29/70  
\end{pmatrix}
\begin{pmatrix}
2/33 \\ 1/99 \cdot \sqrt 3 \\ 41/99
\end{pmatrix}
\begin{pmatrix}
13/198 \\ 1/198 \cdot \sqrt 3 \\ 41/99
\end{pmatrix}
\begin{pmatrix}
7/99 \\ 1/99 \cdot \sqrt 3 \\ 41/99
\end{pmatrix}
\begin{pmatrix}
1/16 \\ 1/128 \cdot \sqrt 3\\ 53/128
\end{pmatrix}
\begin{pmatrix}
17/256 \\ 1/256 \cdot \sqrt 3 \\ 53/128
\end{pmatrix}, \quad \iota(Q_{11})
\]

\[
Q_{12}=\begin{pmatrix}
3/70 \\ 1/35 \cdot \sqrt 3 \\ 29/70
\end{pmatrix}
\begin{pmatrix}
1/20 \\ 3/140 \cdot \sqrt 3 \\  29/70
\end{pmatrix}
\begin{pmatrix}
1/33 \\ 1/33 \cdot \sqrt 3 \\ 41/99
\end{pmatrix}
\begin{pmatrix}
4/99 \\ 1/33 \cdot \sqrt 3\\ 41/99
\end{pmatrix}
\begin{pmatrix}
1/32 \\ 1/32 \cdot \sqrt 3 \\ 53/128
\end{pmatrix}
\begin{pmatrix}
7/256  \\ 7/256 \cdot \sqrt 3 \\ 53/128
\end{pmatrix}, \quad \iota(Q_{12}).
\]
}
The 16 newly-appeared maximal domain $P$ at parameter $t=t_{2,3}$ mentioned in section 4.3 are listed as follows:
{\tiny
\[
P_0 = \begin{pmatrix}
-16/17 \\ 7/17 \cdot \sqrt 3 \\ 7/17
\end{pmatrix}
\begin{pmatrix}
-15/17 \\ 7/17 \cdot \sqrt 3 \\ 7/17
\end{pmatrix}
 \begin{pmatrix}
-10/11 \\ 9/22 \cdot \sqrt 3 \\ 9/22
\end{pmatrix}
 \begin{pmatrix}
-21/22 \\ 9/22 \cdot \sqrt 3 \\ 9/22
\end{pmatrix}
 \begin{pmatrix}
-9/10 \\ 3/10 \cdot \sqrt 3 \\ 2/5
\end{pmatrix}, \quad \iota(P_0)
\]

\[
P_1 = \begin{pmatrix}
-16/17 \\ 7/17 \cdot \sqrt 3 \\ 7/17
\end{pmatrix}
\begin{pmatrix}
-31/34 \\ 13/34 \cdot \sqrt 3 \\ 7/17
\end{pmatrix}
\begin{pmatrix}
-21/22 \\ 4/11 \cdot \sqrt 3 \\ 9/22
\end{pmatrix}
\begin{pmatrix}
-43/44 \\ 17/44 \cdot \sqrt 3 \\ 9/22
\end{pmatrix}
\begin{pmatrix}
-9/10 \\ 3/10 \cdot \sqrt 3 \\ 2/5
\end{pmatrix}, \quad \iota(P_1)
\]

\[
P_2 = \begin{pmatrix}
-16/17 \\ 7/17 \cdot \sqrt 3 \\ 7/17
\end{pmatrix}
\begin{pmatrix}
-43/44 \\ 17/44 \cdot \sqrt 3 \\ 9/22
\end{pmatrix}
\begin{pmatrix}
-21/22 \\ 9/22 \cdot \sqrt 3 \\ 9/22
\end{pmatrix}
\begin{pmatrix}
-53/54 \\ 7/18 \cdot \sqrt 3 \\ 11/27
\end{pmatrix}
\begin{pmatrix}
-9/10 \\ 3/10 \cdot \sqrt 3 \\ 2/5
\end{pmatrix}, \quad \iota(P_0)
\]

\[
P_3 = \begin{pmatrix}
-2/17 \\ 1/17 \cdot \sqrt 3 \\ 7/17
\end{pmatrix}
\begin{pmatrix}
-3/17 \\ 1/17 \cdot \sqrt 3 \\ 7/17
\end{pmatrix}
\begin{pmatrix}
-2/11 \\ 1/11 \cdot \sqrt 3 \\ 9/22
\end{pmatrix}
\begin{pmatrix}
-3/22 \\ 1/11 \cdot \sqrt 3 \\ 9/22
\end{pmatrix}
\begin{pmatrix}
-5/44 \\ 3/44 \cdot \sqrt 3 \\ 9/22
\end{pmatrix} 
\begin{pmatrix}
-7/54 \\ 5/54 \cdot \sqrt 3 \\ 11/27
\end{pmatrix} 
\begin{pmatrix}
-3/10 \\ 1/10 \cdot \sqrt 3 \\ 2/5
\end{pmatrix}, \quad \iota(P_3)
\]

\[
P_4 = \begin{pmatrix}
4/17 \\ 3/17 \cdot \sqrt 3 \\ 7/17
\end{pmatrix}
\begin{pmatrix}
9/34 \\ 5/34 \cdot \sqrt 3 \\ 7/17
\end{pmatrix}
\begin{pmatrix}
2/11 \\ 2/11 \cdot \sqrt 3 \\ 9/22
\end{pmatrix}
\begin{pmatrix}
5/22 \\ 2/11 \cdot \sqrt 3 \\ 9/22
\end{pmatrix}
\begin{pmatrix}
1/4 \\ 7/44 \cdot \sqrt 3 \\ 9/22
\end{pmatrix}
\begin{pmatrix}
5/27 \\ 5/27 \cdot \sqrt 3 \\ 11/27
\end{pmatrix} 
\begin{pmatrix}
1/10 \\ 1/10 \cdot \sqrt 3 \\ 2/5
\end{pmatrix}, \quad \iota(P_4)
\]
}

{\tiny
\[
P_5 = \begin{pmatrix}
-23/34 \\ 15/34 \cdot \sqrt 3 \\ 7/17
\end{pmatrix}
\begin{pmatrix}
-15/22 \\ 9/22 \cdot \sqrt 3 \\ 9/22
\end{pmatrix}
\begin{pmatrix}
-7/11 \\ 9/22 \cdot \sqrt 3 \\ 9/22
\end{pmatrix}
\begin{pmatrix}
-29/44 \\ 19/44 \cdot \sqrt 3 \\ 9/22
\end{pmatrix}
\begin{pmatrix}
-2/3 \\ 11/27 \cdot \sqrt 3 \\ 11/27
\end{pmatrix}, \quad \iota(P_5)
\]

\[
P_6 = \begin{pmatrix}
-21/34 \\ 11/34 \cdot \sqrt 3 \\ 7/17
\end{pmatrix}
\begin{pmatrix}
-29/44 \\ 15/44 \cdot \sqrt 3 \\ 9/22
\end{pmatrix}
\begin{pmatrix}
-15/22 \\ 7/22 \cdot \sqrt 3 \\ 9/22
\end{pmatrix}
\begin{pmatrix}
-7/11 \\ 7/22 \cdot \sqrt 3 \\ 9/22
\end{pmatrix}
\begin{pmatrix}
-2/3 \\ 1/3 \cdot \sqrt 3 \\ 11/27
\end{pmatrix}, \quad \iota(P_6)
\]

\[
P_7 = \begin{pmatrix}
7/17 \\ 1/17 \cdot \sqrt 3 \\ 7/17
\end{pmatrix}
\begin{pmatrix}
4/11 \\ 1/22 \cdot \sqrt 3 \\ 9/22
\end{pmatrix}
\begin{pmatrix}
17/44 \\ 1/44 \cdot \sqrt 3 \\ 9/22
\end{pmatrix}
\begin{pmatrix}
9/22 \\ 1/22 \cdot \sqrt 3 \\ 9/22
\end{pmatrix}
\begin{pmatrix}
10/27 \\ 1/27 \cdot \sqrt 3 \\ 11/27
\end{pmatrix}, \quad \iota(P_7)
\]
}

The 16 newly-appeared maximal domain $Q$ in $\mathcal Z(A_{2,4})$ at $s_{2,3}$ are the following:
{\tiny
\[
Q_0 = \begin{pmatrix}
-98/99 \\ 7/99 \cdot \sqrt 3 \\ 41/99
\end{pmatrix}
\begin{pmatrix}
-97/99 \\ 7/99 \cdot \sqrt 3 \\ 41/99
\end{pmatrix}
\begin{pmatrix}
-63/64 \\ 9/128 \cdot \sqrt 3 \\ 53/128
\end{pmatrix}
\begin{pmatrix}
-127/128 \\ 9/128 \cdot \sqrt 3 \\ 53/128
\end{pmatrix}
\begin{pmatrix}
-57/58 \\ 3/58 \cdot \sqrt 3 \\ 12/29
\end{pmatrix}, \quad \iota(Q_0)
\]

\[
Q_1=\begin{pmatrix}
-98/99 \\ 7/99 \cdot \sqrt 3 \\ 41/99
\end{pmatrix}
\begin{pmatrix}
-65/66 \\ 13/198 \cdot \sqrt 3 \\ 41/99
\end{pmatrix}
\begin{pmatrix}
-1 \\ 17/256 \cdot \sqrt 3 \\ 53/128
\end{pmatrix}
\begin{pmatrix}
-127/128 \\ 1/16 \cdot \sqrt 3 \\ 53/128
\end{pmatrix}
\begin{pmatrix}
-57/58 \\ 3/58 \cdot \sqrt 3 \\ 12/29
\end{pmatrix}, \quad \iota(Q_1)
\]

\[
Q_2=\begin{pmatrix}
-98/99 \\ 7/99 \cdot \sqrt 3 \\ 41/99
\end{pmatrix}
\begin{pmatrix}
-1 \\ 17/256 \cdot \sqrt 3 \\ 53/128
\end{pmatrix}
\begin{pmatrix}
-127/128\\ 9/128 \cdot \sqrt 3 \\ 53/128
\end{pmatrix}
\begin{pmatrix}
-127/128 \\ 21/314 \cdot \sqrt 3 \\ 65/157
\end{pmatrix}
\begin{pmatrix}
-57/58 \\ 3/58 \cdot \sqrt 3 \\ 12/29
\end{pmatrix}, \quad \iota(Q_2)
\]

\[
Q_3=\begin{pmatrix}
-61/99 \\ 25/99 \cdot \sqrt 3 \\ 41/99
\end{pmatrix}
\begin{pmatrix}
-20/33 \\ 25/99 \cdot \sqrt 3 \\ 41/99
\end{pmatrix}
\begin{pmatrix}
-79/128 \\ 33/128 \cdot \sqrt 3 \\ 53/128
\end{pmatrix}
\begin{pmatrix}
-39/64 \\ 33/128 \cdot \sqrt 3 \\ 53/128
\end{pmatrix}
\begin{pmatrix}
-191/314 \\ 81/314 \cdot \sqrt 3 \\ 65/157
\end{pmatrix}
\begin{pmatrix}
-37/158 \\ 15/58 \cdot \sqrt 3 \\ 12/29
\end{pmatrix}, \quad (Q_3)
\]

\[
Q_4=\begin{pmatrix}
4/99 \\ 1/13 \cdot \sqrt 3 \\ 41/99
\end{pmatrix}
\begin{pmatrix}
1/22 \\ 5/198 \cdot \sqrt 3 \\ 41/99
\end{pmatrix}
\begin{pmatrix}
1/32 \\ 1/32 \cdot \sqrt 3 \\ 53/128
\end{pmatrix}
\begin{pmatrix}
11/256 \\ 7/256 \cdot \sqrt 3 \\ 53/128
\end{pmatrix}
\begin{pmatrix}
5/128 \\ 1/32 \cdot \sqrt 3 \\ 53/128
\end{pmatrix}
\begin{pmatrix}
5/157 \\ 5/157 \cdot \sqrt 3 \\ 65/157
\end{pmatrix}
\begin{pmatrix}
1/58 \\ 1/58 \cdot \sqrt 3 \\ 12/29
\end{pmatrix}, \quad \iota(Q_4).
\]

\[
Q_5=\begin{pmatrix}
-17/18 \\ 5/66 \cdot \sqrt 3 \\ 41/99
\end{pmatrix}
\begin{pmatrix}
-15/16 \\ 9/128 \cdot \sqrt 3 \\ 53/128
\end{pmatrix}
\begin{pmatrix}
-121/128 \\ 9/128 \cdot \sqrt 3 \\ 53/128
\end{pmatrix}
\begin{pmatrix}
-241/256 \\ 19/256 \cdot \sqrt 3 \\ 53/128
\end{pmatrix}
\begin{pmatrix}
-148/157 \\ 11/157 \cdot \sqrt 3 \\ 65/157
\end{pmatrix}, \quad \iota(Q_5)
\]

\[
Q_6=\begin{pmatrix}
-137/198 \\ 59/198 \cdot \sqrt 3 \\ 41/99
\end{pmatrix}
\begin{pmatrix}
-45/64 \\ 19/64 \cdot \sqrt 3 \\ 53/128
\end{pmatrix}
\begin{pmatrix}
-179/256 \\ 77/256 \cdot \sqrt 3 \\ 53/128
\end{pmatrix}
\begin{pmatrix}
-89/128 \\ 19/64 \cdot \sqrt 3 \\ 53/128
\end{pmatrix}
\begin{pmatrix}
-110/157 \\ 47/157 \cdot \sqrt 3 \\ 65/157
\end{pmatrix}, \quad \iota(Q_6)
\]

\[
Q_7 = \begin{pmatrix}
-7/99 \\ -1/99 \cdot \sqrt 3 \\ 41/99
\end{pmatrix}
\begin{pmatrix}
-17/256 \\ -1/256 \cdot \sqrt 3 \\ 53/128
\end{pmatrix}
\begin{pmatrix}
-1/16 \\ -1/128 \cdot \sqrt 3 \\ 53/128
\end{pmatrix}
\begin{pmatrix}
-9/128 \\ -1/128 \cdot \sqrt 3 \\ 53/128
\end{pmatrix}
\begin{pmatrix}
-10/157\\ -1/157 \cdot \sqrt 3 \\ 65/157
\end{pmatrix}, \quad \iota(Q_7)
\]
}

\end{document}